\documentclass[10pt, b5paper]{article}
\usepackage[utf8]{inputenc}

\usepackage{amsmath,amsthm,amssymb}
\usepackage{mathtools, tikz-cd}
\usepackage{color}
\usepackage{tabularx}

\newcommand{\D}{\operatorname{d}}

\newcommand\blfootnote[1]{%
  \begingroup
  \renewcommand\thefootnote{}\footnote{#1}%
  \addtocounter{footnote}{-1}%
  \endgroup
}

\newtheorem{theorem}{Theorem}[section]
\newtheorem{corollary}{Corollary}[theorem]
\newtheorem{lemma}[theorem]{Lemma}
\newtheorem{proposition}[theorem]{Proposition}

\title{Pseudo-Riemannian and Hessian Geometry Related to Monge-Ampère Structures}

\author{
S. Hronek\footnote{Department of Theoretical Physics and Astrophysics, Faculty of Science, Masaryk University
611 37 Brno, Czech Republic} \\ R. Suchánek\footnote{Dept. of Mathematics and Statistics, Faculty of Science, Masaryk University
611 37 Brno, Czech Republic} \footnote{Angevin Laboratory of Mathematical Research - UMR CNRS 6093, University of Angers, 2 Boulevard Lavoisier, 49045, Angers CEDEX 0, France.}
}

\begin{document}

\maketitle

\noindent \textsc{Abstract. } We study properties of pseudo-Riemannian metrics corresponding to Monge-Ampère structures on four dimensional $T^*M$. We describe a family of Ricci flat solutions, which are parametrized by six coefficients satisfying the Plücker embedding equation. We also focus on pullbacks of the pseudo-metrics on two dimensional $M$, and describe the corresponding Hessian structures. 



\section{Introduction}

Let $T^*M \xrightarrow{\pi} M$ be the cotangent bundle over a real, smooth, two dimensional manifold $M$. Let $\Omega \in \Omega^2(T^*M)$ be the canonical symplectic form, which in the Darboux (or canonical) coordinates is written as
    \begin{align}\label{eq: symplectic form}
        \Omega = \D x\wedge \D p + \D y\wedge \D q \ .
    \end{align}
Let us further consider a $2$-form $\alpha \in \Omega^2(T^* M)$ given by
\begin{align}\label{def: general alpha for 2D SMAE}
        \alpha = A \D p \wedge \D y + B( \D x \wedge \D p - \D y \wedge \D q) + C \D x \wedge \D q + D \D p \wedge \D q + E \D x \wedge \D y \ , 
    \end{align} 
where $A,B,C,D,E \in C^\infty (T^*M)$ are smooth functions. If $\Omega \wedge \alpha = 0$, then the pair $(\Omega, \alpha)$ is called a Monge-Ampère (M-A) structure over $T^*M$ \cite{kushner_lychagin_rubtsov_2006, Lychagin}. This terminology reflects the fact that the pairs $(\Omega, \alpha)$ are in correspondence with a~subclass of all nonlinear second-order PDEs, called Monge-Ampère equations, in the following way. Choose a function $f \in C^\infty (M)$. The deRham differential of $f$ gives rise to a section, $\D f \colon M \to T^*M$, and one can consider the pullback $(\D f )^* \alpha \in \Omega^2(M)$. Then the equation
    \begin{align}\label{def: M-A equation}
        (\D f )^* \alpha = 0
    \end{align}
defines a nonlinear second order PDE with respect to $f$, where the nonlinearity is given by the determinant of the Hessian matrix of $f$. In the above chosen coordinates, the equation \eqref{def: M-A equation} writes
    \begin{align}\label{eq: M-A equation in coordinates}
        A f_{xx} + 2B f_{xy} + C f_{yy} + D \left( f_{xx} f_{yy} - {f_{xy}}^2 \right) + E  = 0  \ ,
    \end{align}
where $A,\ldots,E$ are the coefficients given in \eqref{def: general alpha for 2D SMAE}, but now depending on $x,y,f_x, f_y$, instead of $x,y,p,q$. The equation \eqref{eq: M-A equation in coordinates} is called a \textit{2D (symplectic) Monge-Ampère equation} (shortly just M-A equation). M-A equations arise and have rich applications, for example, in differential geometry of surfaces, integrability of geometric structures, hydrodynamics, acoustics, variational calculus, Riemannian, CR, or complex geometry \cite{BANOS20112187, kushner_lychagin_rubtsov_2006, Lychagin, ASENS_1993_4_26_3_281_0, Rubtsov2019}. For a detailed exposition of some of these applications and geometric treatment of M-A equations, especially in 2D and 3D, see \cite{Kosmann-Schwarzbach2010, kushner_lychagin_rubtsov_2006}.  

From a different perspective, M-A structure $(\Omega, \alpha)$ yields other geometric structures on $T^*M$, for example, complex, product, Kähler, or nearly Calabi-Yau structures (depending on various assumptions on the coefficients $A,\ldots,E$) \cite{BANOS2007841, BANOS20112187, Kosmann-Schwarzbach2010,kushner_lychagin_rubtsov_2006, Lychagin, Rubtsov_1997, Rubtsov2019}. We call a M-A structure \textit{non-degenerate}, if the \textit{Pfaffian of a M-A structure}, which is given by the equation
    \begin{align}\label{eq: Pfaffian}
        \alpha \wedge \alpha = \operatorname{Pf}(\alpha) \Omega \wedge \Omega  \ ,
    \end{align}
satisfies $ \operatorname{Pf}(\alpha) \neq 0$. The sign of the Pfaffian decides whether M-A structure gives rise to a complex or product structure, and whether the M-A equation corressponding to the M-A structure is elliptic or hyperbolic \cite{BANOS2007841, kushner_lychagin_rubtsov_2006, Lychagin}. We have found that non-degeneracy of M-A structures on $T^*M$ is equivalent to non-degeneracy of certain bilinear forms on $M$.

Inspired by applications in theoretical meteorology \cite{VolRoul2015, ASENS_1993_4_26_3_281_0, RoubRoul2001}, we are interested in a specific family of symmetric bilinear forms, parameterized by $2$-forms $\alpha$ as given in \eqref{def: general alpha for 2D SMAE} (or, equivalently, by the corresponding coefficients $A, \ldots, E$). More concretely, there is a map
    \begin{align*}
        \Omega^2(T^* M) & \to S^2(T^*M) \ , \\
        \alpha & \mapsto g_\alpha \ ,
    \end{align*}
given locally by 
    \begin{align}\label{eq: L-R def}
         g_\alpha (X,Y) : = \frac{2(\iota_X \alpha \wedge \iota_Y \Omega + \iota_Y \alpha \wedge \iota_X \Omega) \wedge \pi^*\operatorname{vol} }{\Omega \wedge \Omega} \ ,
    \end{align}
where $X,Y \in \Gamma (TT^*M)$, $\iota$ is the interior product, and $\pi^* \operatorname{vol} \in \Omega^2(T^*M)$ is the pullback of a locally chosen top form $\operatorname{vol} \in \Omega^2(M)$ along the cotangent bundle projection. Under a mild assumption on $\alpha$, the symmetric form $g_\alpha$ is a~pseudo-metric on $T^*M$, which is called \textit{Lychagin-Rubtsov metric} \cite{Banos2006IntegrableGA, BANOS20112187, phdthesis}.

The L-R metrics are the main objects of our interest. They find applications, for example, in theoretical meteorology \cite{VolRoul2015, roulstone2009kahler, Rubtsov_1997, Rubtsov2019}. In particular, we want to mention a paper in progress of I. Roulstone, V. Rubtsov, and M. Wolf, where an approach motivated by general principles of topological fluid dynamics is employed. Specifically, the curvature of L-R metric \eqref{eq: L-R def} is used to study geometric aspects of semi-geostrophic incompressible flows in 2D and 3D, which are associated with the accumulation of vorticity of the flow. These results were discussed on a series of lectures \textit{Monge–Ampère Geometry and the Navier–Stokes Equations}, given by I. Roulstone during \textit{Winter School and Workshop Wisla 22} held in the beginning of February 2022. 

\subsection{Stucture of the paper}
After the introduction we focus on properties of L-R metrics over the 4D cotangent bundle $T^*M$ and investigate the Ricci flat case. Using the Plücker embedding, we show that there is a correspondence between the set given by Ricci flat $g_\alpha$ and points in projective space. We then consider the pullback of $g_\alpha$ along sections of $T^*M \to M$ and describe their basic features, as well as some particular cases, among which are Hessian structures. Hessian structures are connected to and have applications in, for example, affine differential geometry, homogeneous spaces, cohomology theory, statistical manifolds, string theory, etc. \cite{AMARI20141, hirohiko2007, doi:10.1142/S0129167X04002338}. We also find a special subclass of Hessian structures given by the solutions of certain Monge-Ampère equations. In this case, we compute the Koszul forms and comment on the related Kähler structure and its Ricci tensor. We note that more general M-A structures (i.e. $A,B,C$ not necessarily zero) naturally lead to deformations of the Hessian structures on $M$. We show that the pullback metrics behave rather independently of the original ones, and that the non-degeneracy of a special subset of the pullback metrics is equivalent to non-degeneracy of the corresponding M-A structures.

\section{Pseudo-Riemannian Lychagin-Rubtsov metric on $T^* M$}

We start with some basic information about the L-R metric. In the canonical coordinates, the matrix of $g_\alpha$ defined by \eqref{eq: L-R def} is 
    \begin{align}\label{eq: L-R symmetric tensor}
        G_\alpha = 
        \begin{pmatrix}
        2C & -2B & D & 0 \\
        -2B & 2A & 0 & D \\
        D & 0 & 0 & 0 \\
        0 & D & 0 & 0
        \end{pmatrix} \ .
    \end{align}
We note the independency of $G_\alpha$ on $E$. From $\det G_\alpha = D^4$ further follows the non-degeneracy condition is 
    \begin{align*}
        \det G_\alpha \neq 0 \iff D \neq 0 \ .
    \end{align*}

\begin{lemma}\label{lemma: basic info about L-R metric}
Let $g_\alpha$ be the L-R metric given by \eqref{eq: L-R def}. Suppose that $\alpha$ given by \eqref{def: general alpha for 2D SMAE} satisfies $D \neq 0$. Then $g_\alpha$ is a pseudo-Riemannian metric field on $T^*M$ with signature $(2,2)$.  
\end{lemma}

\begin{proof}
From \eqref{eq: L-R symmetric tensor} we see that the vector fields  $\frac{\partial}{\partial p}, \frac{\partial}{\partial q}$ span a $2$-dimensional totally isotropic subspace. Since $\dim T_p T^*M = 4$, the signature of $g_\alpha$ is necessarily $\left(2,2 \right)$.
\end{proof}

\subsection{L-R metric and Ricci flatness condition}

We search for the conditions on $\alpha$ so that the L-R metric \eqref{eq: L-R def} satisfies the Ricci flatness condition
\begin{align}\label{eq: Ricci flat condition}
R_{ij} = 0 \ . 
\end{align}
The Ricci curvature $R_{ij}$ and the Ricci scalar $R$ of $g_\alpha$ are, in general, complicated expressions with the second order derivatives of the coefficients $A, B, \ldots$ We focused on the special case $A=B=C=0$. This means that $\alpha$ reduces to  
    \begin{align}\label{eq: alpha with A=B=C=0}
        \alpha = D \D p \wedge \D q + E \D x \wedge \D y \ , 
    \end{align}
and the corresponding Monge-Ampère equation is
    \begin{align*}
        D \det \operatorname{Hess}(f) = - E \ .
    \end{align*}
Although we have restricted our considerations significantly by this choice of $A,B,C$, the above M-A equation has interesting properties. For example, it naturally emerges in the context of incompressible fluid dynamics, and is related to rich geometric structures \cite{kushner_lychagin_rubtsov_2006, Kosmann-Schwarzbach2010, VolRoul2015, Rubtsov2019}. Moreover, the pullback of the L-R metric yields a Hessian structure on $M$, which will be further discussed in the next section. We proceed with the following lemma, which will be used in the subsequent result. 

\begin{lemma}\label{lemma: Ricci scalar for A,B,C = 0}
Let $g_\alpha $ be the L-R metric with $A = B = C = 0$. Then the Ricci scalar is
    \begin{align}\label{eq: Ricci scalar for  A,B,C = 0}
        R=\frac{3}{D^3}\left(-D_qD_y+2DD_{yq}-D_pD_x+2DD_{xp}\right) \ . 
    \end{align}
\end{lemma}

\begin{proof}
The proof is given by a direct computation. Note that the factor $\frac{3}{D^3}$ is well-defined, since $G_\alpha$ corresponds to a metric only if $D \neq 0$, due to lemma \eqref{lemma: basic info about L-R metric}. 
\end{proof}

\begin{proposition}\label{proposition 1}
Let $\alpha \in \Omega^2 (T^*M)$ be given by \eqref{eq: alpha with A=B=C=0}, where $E \in C^\infty (T^*M)$ is arbitrary. Then $g_\alpha$ is a Ricci flat pseudo-Riemannian metric if and only if, 
    \begin{align}\label{einst}
        D(x,y,p,q)=\left(c_1+c_2x+c_3y+c_4p+c_5q+c_6(xp+yq)\right)^{-2} \ ,
    \end{align}
where the constants $c_i \in \mathbb{R}$, $i = 1,2 ,\ldots, 6$, satisfy
    \begin{align}\label{cond}
        c_2c_4+c_3c_5-c_1c_6=0 \ .
    \end{align}
\end{proposition}

\begin{proof}
The Ricci curvature is a symmetric tensor and thus, in general, has $10$ independent components. We will split the set of equations corresponding to $R_{ij} = 0, 1 \leq i,j \leq 4 $ into three subsets. The first subset contains four equations
    \begin{align}\label{intermediate step 2}
        3D_i^2-2DD_{ii} & = 0,\quad i \in \{x,y,p,q\} \ , 
    \end{align}
as well as the second subset
    \begin{align}\label{intermediate step 3}
        \begin{split}
        3D_xD_y-2DD_{xy}&=0 \ , \\
        3D_xD_q-2DD_{xq}&=0 \ , \\
        3D_pD_y-2DD_{py}&=0 \ , \\
        3D_pD_q-2DD_{pq}&=0 \ .
        \end{split}
    \end{align}
The last two equations are
    \begin{align}\label{intermediate step 4}
        \begin{split}
        -2DD_{yq}+3D_xD_p-4DD_{xp}&=0 \ , \\
        -2DD_{xp}+3D_yD_q-4DD_{yq}&=0   \ .    
        \end{split} 
    \end{align}
Now we use the following scheme for looking for solutions of the above system of PDEs. We start with the first equation $3D_x^2-2DD_{xx}=0$, which can be easily transformed by multiplying both sides of the equation with the factor $\frac{1}{D D_x}$ 
    \begin{align*}
        \frac{3}{2}\frac{D_x}{D}=\frac{D_{xx}}{D_x} \ , 
    \end{align*}
thus obtaining an equation, which can be integrated to a separable equation
    \begin{align*}
        \frac{3}{2}\ln D = \ln D_x + c \ . 
    \end{align*}
The solution to the last equation is
    \begin{align*}
        D=\frac{1}{(c_1(y,p,q)+xc_2(y,p,q))^2} \ , 
    \end{align*}
where $c_1(y,p,q),c_2(y,p,q)$ are unknown functions. We use the remaining equations in \eqref{intermediate step 2} and \eqref{intermediate step 3} to fix the dependencies of $c_1(y,p,q)$ and $c_2(y,p,q)$ on the variables $y,p,q$. This leads to 
    \begin{equation*}
        D(x,y,p,q)=\left(c_1+c_2x+c_3y+c_4p+c_5q+c_6(xp+yq)\right)^{-2} \ .
    \end{equation*}
Plugging this in the remaining two equations \eqref{intermediate step 4}, we get the condition \eqref{cond}, which finishes the proof.    
\end{proof}

\subsubsection{Plücker embedding}

Following \cite{Lychagin2019}, let $V$ be a real vector space, $\dim V = 4$. Denote by $\operatorname{Gr} (2,V)$ the Grassmanian of $2D$ subspaces in $V$, and by $\mathbb{P}(\Lambda^2 V) $ the projectivization of the space of exterior $2$-forms on $V$. Let $p_{12}, p_{13}, p_{14}, p_{23}, p_{24}, p_{34}$ be coordinates on $\mathbb{P}(\Lambda^2 V) $. The Plücker embedding is a map 
    \begin{align}\label{Plucker map}
        \operatorname{Gr} (2,V) \to \mathbb{P}(\Lambda^2 V)
    \end{align}
with image given by the equation 
    \begin{align}\label{eq: Plucker embedding} 
        p_{12}p_{34} - p_{13}p_{24} + p_{14}p_{23} = 0 \ .
    \end{align}
Now we are ready to formulate a corollary of proposition \ref{proposition 1}, which describes a link between the solutions $g_\alpha$ and points on quadric in the 5D real projective space.

\begin{corollary}\label{Plucker embedding corollary}
There is a correspondence between the set of metrics $g_\alpha \in  S^2(T^*M)$, which satisfy \eqref{einst}, \eqref{cond}, and the quadric in $\mathbb{R}P^5$ given by the image of the Plücker embedding \eqref{Plucker map}. The condition \eqref{cond} is equivalent to vanishing of the scalar curvature of $g_\alpha$. 
\end{corollary}

\begin{proof}
 Define 
    \begin{align*}
        c_1 = - p_{14} \ , \quad c_2 = p_{12} \ ,  \quad c_3 = - p_{13} \ ,  \quad c_4 = p_{34} \ ,  \quad c_5 = p_{24} \ ,  \quad c_6 = p_{23} \ .
    \end{align*}
Then \eqref{eq: Plucker embedding} is satisfied only if \eqref{cond} holds. If $g_\alpha$ is given by \eqref{einst}, then its scalar curvature is
    \begin{align*}
        R = -24(c_2c_4+c_3c_5-c_1c_6) \ .
    \end{align*}
Thus \eqref{cond} holds only if $R = 0$.
\end{proof}










\section{The pullback metric}

Given a section of the cotangent bundle $\D f \colon M \rightarrow T^*M$, we can pullback the metric from $T^*M$ to $M$ and induce a (pseudo-)Riemannian structure $(\D f)^* g_\alpha$ on $M$. In this section we explore the properties of this metric.



\subsection{General case}

We start with the matrix  of the pullback metric $(\D f)^*g_\alpha$ in canonical coordinates, which we will denote $G^*_\alpha$
    \begin{align}\label{eq: general pullback metric}
        G^*_\alpha =
        \begin{pmatrix}
        2C & -2B\\
        -2B& 2A
        \end{pmatrix}
        +
        2D \operatorname{Hess}(f) \ .
    \end{align}
We remind the reader that we use the same notation for the functions $A,B,C,D \in C^\infty(T^*M)$ (i.e. depending on $x,y,p,q$) as we do for the precompositions with the section of $T^*M \to M$ (i.e. depending on $x,y,f_x, f_y$). The determinant is
    \begin{align}\label{eq: determinant of the pullback metric}
        \det G^*_\alpha  = 4(AC-B^2)+4D\left(Af_{xx}+2Bf_{xy}+Cf_{yy}+Df_{xx}f_{yy}-Df^2_{xy}\right) \ .
    \end{align}
We observe that the non-degeneracy of $g_\alpha$, which was governed by the condition $D \neq 0$, is not directly related to non-degeneracy of $(\D f)^*g_\alpha$, which is much more complicated. This allows for all four situations of degeneracy/non-degeneracy of the pair $\left( g_\alpha, (\D f)^*g_\alpha \right)$. 

\subsection{Special choices of $A, \ldots, E$}
We will now discuss certain choices of the coefficients $A, \ldots, E$, leading to further simplifications and interesting properties of $(\D f)^*g_\alpha$.

\subsubsection{Hessian structures}

Suppose that $A=B=C=0$ and $D=\frac{1}{2}$. Then the pullback metric is of the form
    \begin{align}
        G^*_\alpha =\operatorname{Hess}(f) \ , 
    \end{align}
which is called \textit{a Hessian structure} on $M$. The theory of Hessian structures is well studied, for example, see \cite{hirohiko2007}, where the Riemannian geometry of the Hessian structures is investigated. Interesting properties of these structures are also described, for example, in \cite{AMARI20141, doi:10.1142/S0129167X04002338, hirohiko2007}. 

In \cite{hirohiko2007}, the author defines so-called first and second Kozsul forms, $a_i, b_{ij}$, of the Hessian structure, which are derived from the Christoffel symbols of the corresponding metric as follows
    \begin{align}\label{eq: first and second Koszul forms}
        a_i = \Gamma^k_{ki},\qquad b_{ij}=\partial_j \alpha_i \ .
    \end{align}

We can employ the condition that the M-A equation $ (\D f)^* \alpha = 0$ is satisfied, which yields
    \begin{align*}
        \det G^*_\alpha  = - E \ . 
    \end{align*}
In this case, the first and second Kozsul forms simplify to the following nice form
    \begin{align*}
        a_i & = \frac{1}{2}\frac{\partial \ln \det  G^*_\alpha}{\partial x^i}
        =
        \frac{1}{2}\frac{\partial \ln |E|}{\partial x^i} \ , \\
        b_{ij} & = \frac{1}{2}\frac{\partial^2 \ln |E|}{\partial x^i \partial x^j}=\frac{1}{2}\operatorname{Hess}(\ln |E|) \ .
    \end{align*}

Especially the second Koszul form is of great importance. Following \cite{hirohiko2007}, given a Hessian structure on $M$, one can define a Kähler structure on $TM$. Let $(x^i, \xi^i)$ be a coordinate system on $TM$ and $z^j = x^j + i \xi^j$. Then the Hessian structure is 
    \begin{align*}
        g^T=(g_{ij}\circ \pi) \D z^i \D \overline{z}^j \ .
    \end{align*}
Then the Ricci tensor of $g^T$ is given by the second Koszul form
    \begin{align*}
        R^T_{ij}=-\frac{1}{2} b_{ij}\circ \pi \ .
    \end{align*}
The authors of this paper are interested in the relationship between the complex structure given by the above construction, and the complex structure naturally associated with M-A structures with negative Pfaffian (see \eqref{eq: Pfaffian} for the definition of the notion).

\subsubsection{Deformations of the Hessian structure}

Let us now suppose only $D=\frac{1}{2}$, then 
    \begin{align*}
        G^*_\alpha =
        \begin{pmatrix}
        2C & -2B\\
        -2B& 2A
        \end{pmatrix}
        +
        \operatorname{Hess}(f) \ .
    \end{align*}
It is possible to choose suitable functions $A,B,C$ such that the matrix 
    \begin{align*}
        \begin{pmatrix}
        2C & -2B\\
        -2B & 2A
        \end{pmatrix}
    \end{align*}
is a Hessian matrix of some function $\epsilon g$, where $\epsilon > 0$ is a scalar. Such a choice gives rise to a \textit{deformation of the Hessian structure}, since
    \begin{align*}
       G^*_\alpha = \operatorname{Hess}(f) + \epsilon \operatorname{Hess}(g)=\operatorname{Hess}(f+\epsilon g) \ . 
    \end{align*}
Notably, all the results of the previous paragraphs holds.

\subsubsection{Pfaffian and non-degeneracy}

So far we did not need to assume that the M-A equation related to the M-A structure $(\Omega, \alpha)$ is satisfied by $f$. Nevertheless, if we assume that the M-A equation $(\D f)^* \alpha = 0$ holds, then it allows us to significantly simplify the expression \eqref{eq: determinant of the pullback metric}. 

\begin{proposition}
Let $f \in C^\infty(M)$ and consider a M-A structure $(\Omega, \alpha) \in \Omega^2(T^*M) \times \Omega^2(T^*M)$. Suppose that $f$ satisfies the corresponding M-A equation $(\D f)^* \alpha = 0$. Then $(\D f)^*g_\alpha$ is a metric on $M$ only if $\operatorname{Pf}(\alpha) \neq 0$, which is equivalent to non-degeneracy of the M-A structure $(\Omega, \alpha)$. Moreover, the eigenvalues of the matrix of $(\D f)^*g_\alpha$ in canonical coordinates are
    \begin{align}\label{eq: eigenvalues for the special pullback metric}
        \lambda_{1,2}=C+A+D(f_{xx}+f_{yy})\pm 2\sqrt{\operatorname{Pf}(\alpha)} \ .
    \end{align}
\end{proposition}

\begin{proof}
A coordinate description of the M-A equation $(\D f)^* \alpha = 0$ is \eqref{eq: M-A equation in coordinates}. Thus, if $f$ satisfies the equation, then 
    \begin{align*}
          A f_{xx} + 2B f_{xy} + C f_{yy} + D \left( f_{xx} f_{yy} - {f_{xy}}^2 \right) = -E \ .    
    \end{align*}
This implies that \eqref{eq: determinant of the pullback metric} becomes 
    \begin{align*}
        \det G^*_\alpha =
        4(AC-B^2)-4DE \ .
    \end{align*}
Using the equation \eqref{eq: Pfaffian} and the coordinate descriptions \eqref{eq: symplectic form}, \eqref{def: general alpha for 2D SMAE} of the M-A structure, one can easily compute $\operatorname{Pf}(\alpha) = -B^2 + AC - DE$. Therefore
    \begin{align*}
       \det G^*_\alpha  = 4 \operatorname{Pf}(\alpha) \ .
    \end{align*}
The eigenvalues are solutions of the following equation
    \begin{align}\label{eq: equation for the eqigenvalues}
        \lambda^2-\lambda\left(2C+2A+2Df_{yy}+2Df_{xx}\right)
    -4(B^2-AC)-4DE=0
    \end{align} 
and their determination is a direct computation. 
\end{proof}

We observe that the signature of $(\D f)^*g_\alpha$ depends on two M-A equations. Firstly, to obtain the previous result, we had to assume that $f$ solves the M-A equation $(\D f)^* \alpha = 0$. Secondly, we see that the eigenvalue equation \eqref{eq: equation for the eqigenvalues} contains the Laplacian expression $2D(f_{yy}+f_{xx})$. But vanishing of this expression amounts to the Laplace equation $ f_{yy}+f_{xx} = 0 $, which, in general, is a different M-A equation then $(\D f)^* \alpha = 0$.

\section{Conclusions and Outlook}

Motivated by the results of V. Lychagin et. al. \cite{Lychagin, ASENS_1993_4_26_3_281_0}, and I. Roulstone et. al. \cite{VolRoul2015, RoubRoul2001, roulstone2009kahler}, we were interested in (pseudo-)Riemannian and Hessian structures related to Monge-Ampère structures and the corresponding $2D$ Monge-Ampère equations. More concretely, this paper was focused on the following questions.
    \begin{enumerate}
        \item Are there any non-constant Lychagin-Rubtsov metrics with vanishing Ricci curvature? 
        \item If the answer to the first question is positive, can we classify all such Monge-Ampère structures and the corresponding M-A equations?
        \item What kind of geometry on $M$ yields the pullback of the L-R metric along sections of $T^*M \to M$?
    \end{enumerate}
We have presented partial answers to the above questions in the case of 4D metrics on $T^*M$ and their 2D pullbacks on $M$. 

The first question is positively answered in proposition \ref{proposition 1}. There is a family of non-constant L-R metrics, depending on five real parameters satisfying conditions \eqref{einst}, \eqref{cond}, with $R_{ij} = 0$. Considering the second question, we have found definite answers for the case
    \begin{align}\label{eq: det hess f = e}
        D \det \operatorname{Hess} f = - E \ ,
    \end{align}
This amounts to suppressing three out of five degrees of freedom of a general 2D (symplectic) M-A equation \eqref{def: M-A equation} by choosing $A = B = C = 0$. The answer to the last question is also based on this choice, which naturally leads to 2D Hessian structures on $M$, their deformations, and a subclass of Hessian structures determined by the requirement that $f$ is a solution of the M-A equation \eqref{eq: det hess f = e}. We want to emphasize that only some of our results depend on the requirement that $f$ satisfies certain PDE. The L-R metric and its pullback can be introduced without any mentioning of the correspondence between M-A structures and M-A equations, so \eqref{eq: det hess f = e} (or possibly other M-A equation) is an additional assumption. Aside from the above three questions, we observed an interesting relation between a family of solutions of the vacuum field equations and the Plücker embedding, which is described in corollary \ref{Plucker embedding corollary}. Regarding the future work, we are interested in the following. 
    \begin{itemize}
        \item  Further properties of the curvature tensor of $g_\alpha$ and the relation between the curvature and the corresponding M-A equations.
        \item The Hessian structures determined by \eqref{eq: det hess f = e} and aim at establishing further links between M-A structures and Hessian structures (as well as their deformations) in dimensions greater then $2$. 
        \item Comparison of the complex structure related to the Koszul form with the complex structures associated with M-A equations satisfying $\operatorname{Pf}(\alpha) > 0$. 
    \end{itemize}

\section{Acknowledgment}

Both authors are grateful for the valuable comments and suggestions of the reviewer. R. Suchánek is grateful to J. Slovák and V. Rubtsov for their discussions and advices, as well as to the Masaryk University and to the University of Angers for their hospitality. The work of R. Suchánek was financially supported under the project GAČR EXPRO GX19-28628X, and project MUNI/A/1092/2021, and partly by the Barrande Fellowship program organized by The French Institute in Prague (IFP) and the Czech Ministry of Education, Youth and Sports (MYES). Authors are also grateful to the organizers of the 42nd Winter School Geometry and Physics for a wonderful event, which enabled the authors to work on this subject in a very inspiring atmosphere.



\bibliographystyle{plain}

\begin{thebibliography}{10}

\bibitem{AMARI20141}
S.~Amari and J.~Armstrong.
\newblock Curvature of hessian manifolds.
\newblock {\em Differential Geometry and its Applications}, 33:1--12, 2014.

\bibitem{VolRoul2015}
B.~Banos, V.~Rubtsov, and I.~Roulstone.
\newblock {Monge--Amp{\`e}re Structures and the Geometry of Incompressible
  Flows}.
\newblock {\em Journal of Physics A: Mathematical and Theoretical}, 49, 10
  2016.

\bibitem{Banos2006IntegrableGA}
Bertrand Banos.
\newblock Integrable geometries and monge-ampere equations.
\newblock {\em arXiv: Differential Geometry}, 2006.

\bibitem{BANOS2007841}
Bertrand Banos.
\newblock {Monge–Ampère equations and generalized complex geometry— The
  two-dimensional case}.
\newblock {\em Journal of Geometry and Physics}, 57(3):841--853, 2007.

\bibitem{BANOS20112187}
Bertrand Banos.
\newblock Complex solutions of monge–ampère equations.
\newblock {\em Journal of Geometry and Physics}, 61(11):2187--2198, 2011.

\bibitem{phdthesis}
Sylvain Delahaies.
\newblock {\em Complex and contact geometry in geophysical fluid dynamics}.
\newblock PhD thesis, 01 2009.

\bibitem{Kosmann-Schwarzbach2010}
Yvette Kosmann-Schwarzbach and Vladimir Rubtsov.
\newblock {Compatible Structures on Lie Algebroids and Monge-Amp{\`e}re
  Operators}.
\newblock {\em Acta Applicandae Mathematicae}, 109(1):101--135, Jan 2010.

\bibitem{kushner_lychagin_rubtsov_2006}
Alexei Kushner, Valentin Lychagin, and Vladimir Rubtsov.
\newblock {\em Contact Geometry and Nonlinear Differential Equations}.
\newblock Encyclopedia of Mathematics and its Applications. Cambridge
  University Press, 2006.

\bibitem{Lychagin}
V.~V. {Lychagin}.
\newblock {Contact Geometry and Non-Linear Second-Order Differential
  Equations}.
\newblock {\em Russian Mathematical Surveys}, 34(1):149--180, February 1979.

\bibitem{Lychagin2019}
Valentin~V. Lychagin and Volodya Roubtsov.
\newblock {\em Monge--Amp{\`e}re Grassmannians, Characteristic Classes and All
  That}, pages 233--257.
\newblock Springer International Publishing, Cham, 2019.

\bibitem{ASENS_1993_4_26_3_281_0}
Valentin~V. Lychagin, V.~N. Rubtsov, and I.~V. Chekalov.
\newblock A classification of {Monge-Amp\`ere} equations.
\newblock {\em Annales scientifiques de l'\'Ecole Normale Sup\'erieure}, Ser.
  4, 26(3):281--308, 1993.

\bibitem{roulstone2009kahler}
Ian Roulstone, Bertrand Banos, J.D. Gibbon, and V.N. Roubtsov.
\newblock {Kähler geometry and Burgers' vortices}.
\newblock 2009.

\bibitem{RoubRoul2001}
V.~Rubtsov and I.~Roulstone.
\newblock Holomorphic structures in hydrodynamical models of nearly geostrophic
  flow.
\newblock {\em Proc. R. Soc. Lond. A.}, 457:1519–1531, 06 2001.

\bibitem{Rubtsov_1997}
V.~N. Rubtsov and I.~Roulstone.
\newblock {Examples of quaternionic and Kähler structures in Hamiltonian
  models of nearly geostrophic flow}.
\newblock {\em Journal of Physics A: Mathematical and General}, 30(4):L63--L68,
  feb 1997.

\bibitem{Rubtsov2019}
Volodya Rubtsov.
\newblock {\em Geometry of Monge--Amp{\`e}re Structures}, pages 95--156.
\newblock Springer International Publishing, Cham, 2019.

\bibitem{hirohiko2007}
Hirohiko Shima.
\newblock {\em The Geometry of Hessian Structures}.
\newblock WORLD SCIENTIFIC, 2007.

\bibitem{doi:10.1142/S0129167X04002338}
BURT TOTARO.
\newblock {The curvature of a Hessian metric}.
\newblock {\em International Journal of Mathematics}, 15(04):369--391, 2004.

\end{thebibliography}


\end{document}